\documentclass[12pt,article]{amsart}

\usepackage[margin=1in]{geometry}
\usepackage{amsmath, amssymb,amsthm,url,mathrsfs,graphicx,amscd,amsfonts}
\usepackage[mathscr]{eucal}
\usepackage{dsfont}
\usepackage{mathtools}
\usepackage{hyperref}

\usepackage[utf8]{inputenc}
\usepackage{graphicx,todonotes,hyperref}
\usepackage{epstopdf}
\usepackage{inputenc}
\usepackage{enumitem}
\usepackage{amsmath}
\usepackage{graphicx,todonotes,hyperref}
\usepackage{xcolor}
\usepackage{amssymb}
\usepackage{bbm}
\newtheorem*{acknowledgements*}{Acknowledgements}

\allowdisplaybreaks
\newtheorem{theorem}{Theorem}[section]
\newtheorem{lemma}[theorem]{Lemma}
\newtheorem{corollary}[theorem]{Corollary}
\newtheorem{proposition}[theorem]{Proposition}
\newtheorem*{definition}{Definition}

\newtheorem*{question}{Question}
\newtheorem*{theorem*}{Theorem}
\newtheorem*{thm}{Theorem A}

\theoremstyle{remark}\newtheorem{remark}{Remark}

\newtheorem{case}{Case}

\numberwithin{equation}{section}
\newtheorem{example}{Example}

\renewcommand{\phi}{\varphi}

\def\O{\operatorname{O}}

\def\Var{\operatorname{Var}}

\begin{document}
	
	\title[PPC in higher dimensions]{On higher dimensional Poissonian pair correlation}
	
	\author{Tanmoy Bera, Mithun Kumar Das, Anirban Mukhopadhyay}
	\address{$^{1,3}$The Institute of Mathematical Sciences, A CI of 
		Homi Bhabha National Institute, CIT Campus, Taramani, Chennai
		600113, India.}
		\address{$^{2}$ National Institute of Science Education and Research, A CI of 
		Homi Bhabha National Institute, Jatni, Khurda, 
		752050, India.}
	\email[Tanmoy Bera$^1$]{tanmoyb@imsc.res.in}
	\email[Mithun Kumar Das$^2$]{das.mithun3@gmail.com} 
	\email[Anirban Mukhopadhyay$^3$]{anirban@imsc.res.in}

	\maketitle
	
	\begin{abstract}
		In this article we study the pair correlation statistic for higher dimensional sequences. We show that for any $d\geq 2$, strictly 
		increasing sequences $(a_n^{(1)}),\ldots, (a_n^{(d)})$ of natural
		numbers have {\it metric Poissonian pair correlation} with respect to 
		sup-norm if their {\it joint additive energy} is $\O(N^{3-\delta})$
		for any $\delta>0$. Further, in dimension two, we establish 
		an analogous result with respect to the $2$-norm. 
		
		As a consequence, it follows that $(\{n\alpha\}, \{n^2\beta\})$ and $(\{n\alpha\}, 
		\{[n\log^An]\beta\})$ ($A \in [1,2]$) have Poissonian pair correlation 
		for almost
		all $(\alpha,\beta)\in \mathbb{R}^2$ with respect to sup-norm and $2$-norm.
		This gives a negative answer to the question raised by Hofer and  Kaltenb\"ock ~\cite{haltonnotPPC2021}.
		The proof uses estimates for {\it `Generalized' GCD-sums}.
	\end{abstract}
	
	\section{Introduction}
	Let $(x_n)\in[0,1)$ be a sequence, $s>0$ be a real number and $N$ be a natural number. The pair correlation statistic of $(x_n)$ is defined as follows:
	\[R_2(s,(x_n),N) :=\frac{1}{N}\#\Big\{1\leq n\neq m\leq N: \|x_n-x_m\|\leq \frac{s}{N}\Big\},\]
	where for any real $x$, $\displaystyle\|x\|:=\inf_{m\in\mathbb{Z}}|x+m|,$ the nearest integer distance. The sequence $(x_n)$ is said to have Poissonian pair correlation (PPC) if for all $s>0$,\\
	\[\lim_{N\to\infty}R_2(s,(x_n),N)= 2s .\]
	This concept originated from theoretical physics and it plays a crucial role in the Berry–Tabor conjecture.
	Rudnick and Sarnak~\cite{rudnick1998pair} first studied this notion from a mathematical point of view. 
	Since then this topic has received wide
attention~\cite{RSZ2001, RZ1999}, and several generalizations are known (see ~\cite{BZ,hauke2021weak,hinrichs2019multi,marklof2020pair,steinerberger2020poissonian}).
	
	The theory of uniform distribution (or equidistribution) of a sequence has a long history. It is known that Poissonian pair correlated sequences are necessarily uniformly distributed (see ~\cite{aistleitner2018pair,steinerberger2020poissonian,Grepstad2017Larcher}) but the converse is not true.
	
	Only recently, a concept of pair correlation for the higher dimensional sequences have been introduced in \cite{hinrichs2019multi} with respect to sup-norm and in \cite{steinerberger2020poissonian} with respect to $2$-norm. Throughout the article, we assume that $d\geq 2$. For $\boldsymbol{x}=(x^{(1)},\ldots,x^{(d)}) \in \mathbb{R}^d$, we denote $\|\boldsymbol{x}\|_\infty=\max{\{\|x^{(1)}\|,\ldots,\|x^{(d)}\|\}}$ and $\|\boldsymbol{x}\|_2=(\|x^{(1)}\|^2+\cdots +\|x^{(d)}\|^2)^{1/2}$.
	Let $ (\boldsymbol{x}_n)_{n\geq 1}$ be a sequence in $[0,1)^d$. For $s>0$ we write
	\[R_{2,\infty}^{(d)}(s,(\boldsymbol{x}_n),N):=\displaystyle\frac{1}{N}\#\Big\{1\leq m\neq n\leq N: \|\boldsymbol{x}_n-\boldsymbol{x}_m\|_\infty\leq \frac{s}{N^{1/d}}\Big\}\] 
	and \[R_{2,2}^{(d)}(s,(\boldsymbol{x}_n),N):=\displaystyle\frac{1}{N}\#\Big\{1\leq n\neq m\leq N: \|\boldsymbol{x}_n-\boldsymbol{x}_m\|_2\leq \frac{s}{N^{1/d}}\Big\}.\]
	
	\begin{definition}
		A sequence $(\boldsymbol{x}_n)_{n\geq 1}$ in $\mathbb{R}^d$ is said to have $\infty$-PPC if for all $s>0$, 
		\[R_{2,\infty}^{(d)}(s,(\boldsymbol{x}_n),N) \to (2s)^d \text{ as }N\to\infty,\] 
		and $2$-PPC if for all $s>0$,
		\[R_{2,2}^{(d)}(s,(\boldsymbol{x}_n),N) \to w_ds^d \text{ as }N\to\infty,\]
		where $w_d$ is the volume of the unit ball of $\mathbb{R}^d$ in $2$-norm.
	\end{definition}
	Similar to the one-dimensional case, it has been proved in \cite{hinrichs2019multi} that  $\infty$-PPC implies uniform distribution and in \cite{steinerberger2020poissonian} that $2$-PPC implies uniform distribution for higher dimensional sequences.
\begin{remark}
Instead of defining the counting function $R_2$ for balls centred at the
origin, we could have defined it for any balls and consequently defined a
stronger version of PPC by demanding convergence for every ball. In
that case, the notion of
$2$-PPC and $\infty$-PPC would coincide; furthermore, it is going to be
independent of the norm. But there is no clear way
to show this equivalence for our present definitions although the
norms are topologically equivalent.
On the contrary, we believe that such an equivalence does not hold,
though we do not have any example to demonstrate so.
\end{remark}
	
	The purpose of this article is to show $\infty$-PPC and $2$-PPC for some higher dimensional sequences of the form $(\{a_n^{(1)}\alpha_1\},\ldots,\{a_n^{(d)}\alpha_d\})$, where $(a_n^{(i)})$, for $i=1,2,\ldots,d$ are sequences of natural numbers and  $\boldsymbol{\alpha}=(\alpha_1,\ldots,\alpha_d)\in\mathbb{R}^d$. For simplicity, in this case we denote the respective pair correlation statistics by $R_{2,\infty}^{(d)}(s,\boldsymbol{\alpha},N)$ and $R_{2,2}^{(d)}(s,\boldsymbol{\alpha},N)$. 
	\begin{definition}
	We say that a $d$-dimensional sequence $(a_n^{(1)},\ldots,a_n^{(d)})$ has {\it metric Poissonian pair correlation} with respect to sup-norm ($\infty$-MPPC) if \[R_{2,\infty}^{(d)}(s,\boldsymbol{\alpha},N) \rightarrow (2s)^d \mbox{ as } N\rightarrow \infty,\] for almost all $\boldsymbol{\alpha}\in\mathbb{R}^d$. Moreover, $2$-MPPC is defined analogously.
	\end{definition}
	
Our results depend on the notion of additive energy of integer sequences. For a finite subset $A$ of integers, the additive energy $E(A)$ of $A$ is defined by
	\[E(A):=\#\{a,b,c,d\in A: a+b=c+d\}.\]
	
	Recently, Aistleitner et al.~\cite{aistleitner2017additive} proved that  for a strictly increasing sequence $(a_n)$ of natural numbers the sequence $(\{\alpha a_n\})$ has Poissonian pair correlation for almost all $\alpha,$ provided $E(A_N)\ll N^{3-\epsilon}$ for some $\epsilon>0$, where $A_N$ denotes the set of first $N$ elements of $(a_n)$. In \cite{Bloom2019GCDSA}, Bloom and Walker improved their result by relaxing the condition on the upper bound of additive energy.
	Analogous results for some special higher dimensional sequences with respect to sup-norm were established in \cite{hinrichs2019multi}, where the following theorem was specifically proved:
	\begin{thm}
		Let $(a_n)$ be a strictly increasing sequence of natural numbers, $A_N$ denote the first $N$ elements of $(a_n)$ and suppose that
		\[E(A_N)=\O\Big(\frac{N^3}{(\log N)^{1+\epsilon}}\Big), \text{ for some }\epsilon>0,\]
		then for almost all $\boldsymbol{\alpha}=(\alpha_1,\ldots,\alpha_d)\in\mathbb{R}^d$, $(\{a_n\boldsymbol{\alpha}\})=(\{a_n\alpha_1\},\ldots,\{a_n\alpha_d\})$ has $\infty$-PPC.
	\end{thm}
	
	We consider more general sequences, namely, $(\{a_n^{(1)}\alpha_1\},\ldots,\{a_n^{(d)}\alpha_d\})$ and study their pair correlation property. To state our results, we introduce the notion of joint
    additive energy for several increasing sequences of natural numbers. 
	\begin{definition}[Joint additive energy]
		Let $(a_n^{(1)}),(a_n^{(2)}),\ldots,(a_n^{(d)})$ be strictly increasing sequences of natural numbers and $A_N^{(i)}$ denote the first $N$ elements of $(a_n^{(i)})$, for $1\leq i\leq d.$ 
		The joint additive energy $E(A_N^{(1)};\ldots;A_N^{(d)})$ is given by
		\[E\Big(A_N^{(1)};\ldots;A_N^{(d)}\Big)=\#\Big\{1\leq n,m,k,l\leq N: a_n^{(i)}+a_m^{(i)}=a_k^{(i)}+a_l^{(i)}, i=1,\ldots,d\Big\}.
		\]
	\end{definition}
	\noindent Note that joint additive energy is additive energy in higher dimensions.
	In section~\ref{s2} we study joint additive energy in detail and obtain an upper bound of it.
	
	Now, we state our results.
	\begin{theorem}\label{infinity-PPC thm}
		Let $(a_n^{(1)}),(a_n^{(2)}),\ldots,(a_n^{(d)})$ be strictly increasing sequences of natural numbers and $A_N^{(i)}$ denote the first $N$ elements of $(a_n^{(i)})$, for $1\leq i\leq d.$ Suppose that for some $\delta>0$,
		\[E\Big(A_N^{(1)};\ldots;A_N^{(d)}\Big) = \O\left( N^{3-\delta}\right).\]
		Then $(a_n^{(1)},\ldots,a_n^{(d)})$ has $\infty$-MPPC.
	\end{theorem}
	An immediate consequence is the following.
	\begin{corollary}\label{infinity-PPC thm cor}
		Suppose that for some $\delta>0$,
		\[\displaystyle \min_{1\leq i\leq d}E(A_N^{(i)}) = \O\left( N^{3-\delta}\right).\]
		Then $(a_n^{(1)},\ldots,a_n^{(d)})$ has $\infty$-MPPC.
	\end{corollary}

	In \cite{steinerberger2020poissonian}, Steinerberger introduced the notation of $2$-PPC but did not indicate any sequences which satisfy $2$-PPC. In the following theorems we study $2$-PPC for certain sequences.
	\begin{theorem}\label{2-PPC thm1}
		Let $(a_n)$ be a strictly increasing sequence of natural numbers and $A_N$ denote the first $N$ elements of $(a_n).$
		Assume that 
		\[\displaystyle E(A_N)=\O\Big(\frac{N^3}{(\log N)^{1+\epsilon}}\Big),\text{ for some $\epsilon>0$.}\]
		Then for almost all $\boldsymbol{\alpha}=(\alpha_1,\ldots,\alpha_d)\in\mathbb{R}^d$, the sequence $(\{a_n\boldsymbol{\alpha}\})=(\{a_n\alpha_1\},\ldots,\{a_n\alpha_d\})$ has $2$-PPC.
	\end{theorem}
	\begin{theorem}\label{2-PPC thm2}
Let $(a_n),(b_n)$ be strictly increasing sequences of natural numbers and $A_N$, $B_N$ denote their first $N$ elements respectively. Suppose that
		\begin{align*}
			E\Big(A_N; B_N\Big)=\O\Big(N^{3}\exp{\big(-(\log{N})^{\frac{1}{2}+\delta}\big)}\Big),\,\mbox{ for some } \delta>0.
		\end{align*}
		Then $(a_n,b_n)$ has $2$-MPPC.
	\end{theorem}
	
	\begin{remark}Due to technical complexity, we prove Theorem~\ref{2-PPC thm2} only for dimension two,
though the arguments can be extended to higher dimensions.
	\end{remark}
Recently, Hofer and  Kaltenb\"ock ~\cite{haltonnotPPC2021} asked the following question:
\begin{question}
For Poissonian pair correlations in the higher dimensional setting, is it necessary for all its component sequences to have the Poissonian pair correlation?
\end{question}
 The following examples give a negative answer to this question.
	\begin{remark}
		For an arbitrary strictly increasing sequence $(a_n)$ of natural numbers and for any $l\geq 2\in\mathbb{N},$
		$(a_n,n^l)$ has $\infty$-MPPC and $2$-MPPC. 
		In particular, $(\{n\alpha\},\{n^2\beta\})$ has $\infty$-PPC and $2$-PPC for almost all $(\alpha,\,\beta)\in\mathbb{R}^2.$  It is known that $(\{n\alpha\})_{n\in \mathbb{N}}$ does not have PPC for any $\alpha\in \mathbb{R}$(for instance see ~\cite{LS}).
		Thus, the $\infty$-PPC and $2$-PPC in higher dimensions do not imply that each component has PPC.
	\end{remark}
	Next, we obtain another such example in dimension two. 
	\begin{theorem}\label{thm6}
		Let $A \in [1, 2]$ be any real number. Then the sequence $(n, [n(\log n)^A])$ has $\infty$-MPPC and $2$-MPPC.
	\end{theorem}
	Currently, we do not know any result showing that $(\{[n(\log n)^A]\alpha\})$ has PPC for any real $\alpha$. However, from \cite[Corollary 1]{Garaev2004} we obtain the associated additive energy $\ll N^3(\log N)^{1-A}$.
	
	\textbf{Notation:}
	Let $h\geq 2$ be an integer, $a$ be a real number and   $\boldsymbol{x}=(x_1,\ldots,x_h),\boldsymbol{y}=(y_1,\ldots,y_h)\in\mathbb{R}^h.$ 
	\begin{itemize}
		\item Coordinate-wise product $\boldsymbol{x}\boldsymbol{y}=(x_1y_1,\ldots,x_hy_h).$
		\item Inner product $\boldsymbol{x}.\boldsymbol{y}=\sum_{1\leq i\leq h}x_iy_i.$
		\item Fractional part $\{\boldsymbol{x}\}=(\{x_1\},\ldots,\{x_h\}).$
		\item Scalar product $a\boldsymbol{x}=(ax_1,\ldots,ax_h).$
		\item $\Sigma_{\boldsymbol{x}\in\mathbb{Z}^h}^{'}$ will mean
			  sum over $\boldsymbol{x}$ with $x_i\neq 0$ for all $i.$
	\end{itemize}
	Further, for $a, b \in \mathbb{Z}$ their greatest common divisor(GCD) is denoted by $(a, b).$ Define $e(x):=e^{2\pi ix}$ for $x \in \mathbb{R}$. 
	
	Let $h\geq 1$. If $a_n^{(i)}=a_n$ for all $i$ then we denote $A_N^{(i)}$'s by $A_N$ and $E(A_N^{(1)};\ldots;A_N^{(h)})$ by $E(A_N).$
	For $\boldsymbol{v}=(v_1,\cdots,v_h), v_i\in\mathbb{Z}$ is non-zero for all $i$, we define the representation function $\mathcal{R}_N(\boldsymbol{v})$ by
	\begin{align}\label{R_N}
		\mathcal{R}_N(\boldsymbol{v}):=\#\{1\leq m\neq n\leq N: a_n^{(i)}-a_m^{(i)}=v_i,\:1\leq i\leq h\}.
	\end{align}
	We use $\mathcal{R}_N(\boldsymbol{v})$ and $\mathcal{R}_N(v_1,\ldots,v_h)$ interchangeably.
	One can see that, 
	\[\sideset{}{'}\sum_{\boldsymbol{v}\in\mathbb{Z}^h}\mathcal{R}_N(\boldsymbol{v})^2\leq E(A_N^{(1)};\ldots;A_N^{(h)}).\]

	\section{Joint additive energy}\label{jadditive}\label{s2}
	Let $A_N^{(i)}$ denote the first $N$ elements of $(a_n^{(i)})$, for $1\leq i\leq d$. It is easy to see that the joint additive energy satisfies the following trivial estimate 
	\[N^2\leq E(A_N^{(1)};\ldots;A_N^{(d)})\leq\min_{1\leq i\leq d}{E(A_N^{(i)})}\leq N^3.
	\]
	From this observation we conclude that for any strictly increasing sequence $A:=(a_n)_{1\leq n\leq N}$ of $N$ natural numbers and any fixed integer $l\geq 2$, $B_l:=\{n^l: 1\leq n\leq N\}$ we have $E(A;B_l)\ll N^{2+\epsilon}$ for any $\epsilon>0$ (since $E(B_l)\ll N^{2+\epsilon}$, see ~\cite{aistleitner2017additive}).
	
	For $1\leq s\leq d$, Vinogradov's mean value $J_{s,d}(N)$ is the number of solutions in $\mathbb{N}$ of the system:
	\[x_1^{i}+\cdots+x_s^i = y_1^{i}+\cdots+y_s^i, \quad 1\leq i\leq d.\] For $B_l:=\{n^l: 1\leq n\leq N\}$, we see that $E(B_1;B_2;\ldots;B_d)$ is equal to the Vinogradov's mean value $J_{2,d}(N)$. 
	Then from the work of Ford and Wooley~\cite{FW} we get
	\begin{equation*}
		E(B_1;B_2;\ldots;B_d)= J_{2,d}(N) \ll N^{2+\epsilon},
	\end{equation*}
	for any $\epsilon>0$.
	A natural question is how small the joint additive energy can be when all the components have considerably large additive energy. For example, \cite[Theorem 1]{Garaev2001} says that the sequence $([n\log n])$ has additive energy $\gg \frac{N^3}{\log N}$, but we show that the joint additive energy of $(n)$ and 
	$([n\log n])$ is much smaller.
	
	Let $(f(n))$, $(g(n))$ be two strictly increasing sequences of natural numbers, and let $F_N$ and $G_N$ be the sets of their first $N$ elements, respectively. For a given positive integer $l$ with $1\leq l< N$, denote by $Q_l:=Q_l(N)$ the number of solutions of the system of equations
	\begin{align}\label{jointeq1}
		\begin{cases}
			f(x)+f(y)=f(x+l)+f(z)\\
			g(x)+g(y)=g(x+l)+g(z),
		\end{cases}
		\text{where } \,1\leq x<x+l\leq z<y\leq N.
	\end{align}
	We have the following bounds for the joint additive energy:
	\begin{proposition}\label{bound thm}
		For any $0<\epsilon<1$ we have
		\[\frac{N^4}{(f(N)+1)(g(N)+1)} \ll E(F_N;G_N)\ll N^{2+\epsilon}+N^\epsilon\displaystyle\sum_{1\leq l\leq N^{1-\epsilon}}Q_l.\]
	\end{proposition}
	\begin{proof} We adopt the idea of the proof of \cite[Theorem 1]{Garaev2004} and \cite[Theorem 1]{Garaev2001} for the upper and lower bounds, respectively, to obtain these estimates.\\
		For any real $\alpha_1, \alpha_2$ we set 
		\[S(\alpha_1,\alpha_2)=\sum_{1\leq n\leq N}e\big(f(n)\alpha_1+g(n)\alpha_2\big).\] 
		Then, by orthogonality of the exponential function, we obtain
		\begin{equation}\label{4th moment}
			E(F_N;G_N)=\int_{[0,1]^2}|S(\alpha_1,\alpha_2)|^4d\boldsymbol{\alpha}.
		\end{equation}
		For integers $1\leq s\leq[N^\epsilon],$ set
		$I_s:=\{n\in\mathbb{Z}\colon(s-1)N^{1-\epsilon}<n\leq sN^{1-\epsilon}\}$ and $I_{[N^\epsilon]+1}:=\{n\in\mathbb{Z}\colon [N^\epsilon]N^{1-\epsilon}<n\leq N\}.$ Then, by separating diagonal and off-diagonal terms, applying triangular inequality and the partition of $[1,N]$ into $I_s$'s we deduce
		\begin{align*}
			|S(\alpha_1,\alpha_2)|^4 &\ll N^2+\Big|\displaystyle\sum_{1\leq s\leq [N^\epsilon]+1}\sum_{n\in I_s}\sum_{1\leq n<m\leq N}e\big((f(n)+f(m))\alpha_1+(g(n)+g(m))\alpha_2\big)\Big|^2.
		\end{align*}
		Now, by applying Cauchy-Schwarz inequality on the sum over $s$, we get
		\begin{align*}
			|S(\alpha_1,\alpha_2)|^4 \ll N^2+N^\epsilon\displaystyle\sum_{1\leq s\leq [N^\epsilon]+1}\Big|\sum_{n\in I_s}\sum_{1\leq n<m\leq N}e\big((f(n)+f(m))\alpha_1+(g(n)+g(m))\alpha_2\big)\Big|^2.
		\end{align*}
		We apply the above bound of $|S(\alpha_1,\alpha_2)|^4$ in \eqref{4th moment}, expand the squares, and integrate to obtain 
		\begin{align}\label{28/6}
			E(F_N;G_N)\ll N^2+N^\epsilon\displaystyle\sum_{1\leq s\leq [N^\epsilon]+1}\sum_{n,n_1\in I_s}\sum_{\substack{1\leq n<m\leq N\\1\leq n_1<m_1\leq N\\f(n)+f(m)=f(n_1)+f(m_1)\\g(n)+g(m)=g(n_1)+g(m_1)}}1.
		\end{align}
		Observe that $n, n_1 \in I_s$ imply that $|n-n_1|\leq N^{1-\epsilon}$. If $n_1=n$, then $m=m_1$, and in this case the contribution of the sums in the right-hand side of \eqref{28/6} is $N^2$. Otherwise, by writing $|n-n_1|=l$, the contribution of such sums is 
		\[ \ll \sum_{l\leq N^{1-\epsilon}} Q_l.\]
		Combining these estimates with \eqref{28/6}, we obtain the upper bound.
		
		To estimate the lower bound, we start with the identity
		\begin{align*}
			N^2=\int_{[0,1]^2}S^2(\alpha_1,\alpha_2) \sum_{m=0}^{2f(N)}e(-m\alpha_1)  \sum_{n=0}^{2g(N)}e(-n\alpha_2) d\boldsymbol{\alpha}.
		\end{align*}
 The lower bound follows from an application of the Cauchy-Schwarz inequality to the right-hand side.
	\end{proof}
	\begin{theorem}\label{n and F(n) thm}
		Let $f(x)=x$ and $g(x)=[h(x)]$, where $h$ is a real valued function which is three times continuously differentiable on the segment $[1,N]$ with $h^{'}(x)>0$, $h^{''}(x)>0$ and $h^{'''}(x)<0$. For any real $0<\epsilon<1$,
		\[E(F_N;G_N)\ll N^{2+\epsilon}+\frac{N^{1+\epsilon}\log N}{h^{''}(N)}.\]
	\end{theorem}
	The following corollary is immediate.
	\begin{corollary}\label{coro2.3}
		Let $1\leq A \leq 2$ and $0<\epsilon<1$. 
		If $h(n)=n(\log n)^A,$ then \[E(F_N;G_N)\ll N^{2+\epsilon}(\log{N})^{2-A}.\]
	\end{corollary}
	\begin{proof}[Proof of Theorem \ref{n and F(n) thm}]
		Here we take  $f(n)=n$ and $g(n)=[h(n)]$, so $Q_l$ in  \eqref{jointeq1}  reduces to the  number of solutions of the equation
		\begin{align*}
			[h(n)]+[h(m+l)]=[h(n+l)]+[h(m)], \mbox{ with }1\leq n<n+l\leq m<m+l\leq N.
		\end{align*}
		Now, following the proof of \cite[Theorem 1]{Garaev2001} we get
		\[Q_l\ll N\Big(\frac{2}{lh^{''}(N)}+1\Big).
		\]
		Combining this bound with Proposition \ref{bound thm} we obtain the required upper bound.
	\end{proof}
	
	\section{``Generalized" GCD Sums}
	Given a function $f:\mathbb{N}\rightarrow \mathbb{C}$ with finite support and $\alpha \in (0, 1]$, the so-called GCD sum (also
known as G\'{a}l sum) is defined as 
	\[S_f(\alpha):=\sum_{a, b} f(a)\overline{f(b)} \frac{(a, b)^{2\alpha}}{(ab)^\alpha}.\] 
	This sum plays a key role in finding large values of the Riemann zeta function (see \cite{BS,BT} ).  Moreover, it is connected to the
theory of equidistribution and pair correlation (for instance, see \cite{aistleitner2017additive, Bloom2019GCDSA}). 
	We introduce here a new type of GCD sum in higher dimensions. Let $f:\mathbb{N}^d \rightarrow \mathbb{C}$ be a function with finite support and $\alpha \in (0, 1]$. Define $d$ dimensional GCD sum
	\[S_f(d;\alpha):= \displaystyle\sum_{\boldsymbol{a},\boldsymbol{b} \in \mathbb{N}^d}f(\boldsymbol{a})\overline{f(\boldsymbol{b})}\prod_{i=1}^d\frac{(a_i,b_i)^{2\alpha}}{(a_ib_i)^\alpha}.\]
	The next result provides an upper bound for $S_f(d;\alpha)$.
	
	\begin{proposition}\label{gen gcd lemma}
		Let $f:\mathbb{N}^d\to\mathbb{C}$ be an arithmetic function with finite support of cardinality $K.$ Then we have the following estimates:
		\begin{align*}
			S_f(d;\alpha) \ll \begin{cases}
				(\log\log K)^{\O(1)}\|f\|_2^2, & \mbox{ if } \,\alpha =1,\\
				\exp\left(C(\alpha)(\log K)^{1-\alpha}(\log\log K)^{-\alpha}\right)\|f\|_2^2, & \mbox{ if } \,\frac{1}{2}<\alpha<1,
			\end{cases}
		\end{align*}
		where $C(\alpha)$ is an absolute positive constant depending on $\alpha$.
		
		Furthermore, if $f$ is a positive real-valued satisfying $\|f\|_1\geq 3$ and $K\geq\log\|f\|_1,$ then 
		\[ S_{f}\left(d; {1}/{2}\right) \ll \frac{\exp(C(\log K \log\log \|f\|_1)^{1/2})}{(\log \|f\|_1+O(1))^d}\|f\|_2^2,\] 
		for some positive absolute constant $C$.
	\end{proposition}

\begin{example}
	Let $\alpha=1$, $d=2$ and $f=\mathbbm{1}_{[1,N]^2}$.
	Then, 
	\begin{align*}
		S_f(2;1)= \sum_{\boldsymbol{a},\boldsymbol{b} \in [1,N]^2}\prod_{i=1}^2\frac{(a_i,b_i)^{2}}{a_ib_i} = \left(\sum_{a,b\in [1,N]}\frac{(a,b)^2}{a b}\right)^2.
	\end{align*}
	Then by applying a result of G\'al~\cite{gel}, we have 
	\begin{align*}
		S_f(2;1)\ll N^2(\log \log {N})^4.
	\end{align*}
\end{example}
	
	For simplicity, we prove Proposition~\ref{gen gcd lemma} for the case $d=2$ by extending an idea of a random model given by Lewko and Radziwi{\l\l} in \cite{lewko2017refinements}. For $d\geq 3$, the proof is a straightforward extension of the method.
	
	Let $\{X(p): p \text{ prime}\}$ and $\{Y(p): p \text{ prime}\}$ be two collections of independent random variables uniformly distributed on $\mathbb{S}^1$. Also, assume that $\{X(p),Y(p): p\text{ prime}\}$ is an independent collection. For every $n\in\mathbb{N}$, we define $X(n):=\prod_{p^a\parallel n}X(p)^a$ and similarly $Y(n):=\prod_{p^a\parallel n}Y(p)^a$. The random zeta function associated with $X(n)$ is defined by 
	\[\zeta_X(\alpha):=\displaystyle\sum_{n\geq 1}\frac{X(n)}{n^\alpha}, \text{ where $\alpha>1/2,$} \]
	and $\zeta_Y(\alpha)$ is defined similarly.
	For fixed $\alpha>1/2$, these series converge a.e. by Kolmogorov three series theorem (see \cite[Theorem ~15.51]{K}). 
	
	We recall the moment estimates of $\zeta_X(\alpha)$ from Lemma~ 7 of \cite{Bloom2019GCDSA} and use them to prove Proposition~\ref{gen gcd lemma}.
	\begin{lemma}\label{moment lemma} For a real number $l$,
		\begin{equation*}
			\log\mathbb{E}\big[|\zeta_X(\alpha)|^{2l}\big]\ll\begin{cases}
				l\log\log l, \quad & \mbox{if }\,\alpha=1, \\
				C'(\alpha)l^{1/\alpha}(\log l)^{-1}, \quad &\mbox{if }\,  1/2<\alpha<1,\\
				l^2\log((\alpha-1/2)^{-1}), \quad & \mbox{if }\, \frac{1}{2}<\alpha,
			\end{cases}
		\end{equation*}
		where $C'(\alpha)$ is a positive constant, $l\geq 3$ for first two cases and $l\geq 1$ for the final case.
	\end{lemma}
	
	\begin{proof}[Proof of Proposition~\ref{gen gcd lemma}]
		Let us start by defining the double sum
		$$D(X,Y):=\displaystyle\sum_{a,b}f(a,b)X(a)Y(b).$$
		Then, we look at the expectation of $|\zeta_X(\alpha)\zeta_Y(\alpha)D(X,Y)|^2,$ expressed as
		\begin{align*}
			\mathbb{E}\big[|\zeta_X(\alpha)\zeta_Y(\alpha)D(X,Y)|^2\big]=& \displaystyle\sum_{m_1,n_1,m_2,n_2,a,b,c,d}\frac{f(a,b)\overline{f(c,d)}}{m_1^\alpha n_1^\alpha m_2^\alpha n_2^\alpha}\mathbbm{1}_{n_1a=n_2c}\mathbbm{1}_{m_1b=m_2d}.
		\end{align*}
		Now, indicator functions allow us to write 
		$$n_1=\frac{h_1c}{(a,\, c)},\, n_2=\frac{h_1a}{(a,\, c)} \text{ and } m_1=\frac{h_2d}{(b,\, d)}, \,m_2=\frac{h_2b}{(b,\, d)}$$ for positive integers $h_1, h_2$. Thus,
		\begin{align}\label{eq0}
			\mathbb{E}\big[|\zeta_X(\alpha)\zeta_Y(\alpha)D(X,Y)|^2\big]
			=& \sum_{h_1,h_2,a,b,c,d}f(a,b)\overline{f(c,d)}\frac{(a,c)^{2\alpha}(b,d)^{2\alpha}}{(acbd)^\alpha}\frac{1}{h_1^{2\alpha}h_2^{2\alpha}}\\
			\nonumber	= &  \displaystyle\zeta(2\alpha)^2S_f(2;\alpha).
		\end{align}
		Also, we notice that $\mathbb{E}\big[|D(X,Y)|^2\big]=\|f\|_2^2$.
		Let $V$ and $l$ be positive real parameters to be chosen later. Consider the event $\mathcal{A}=(|\zeta_X(\alpha)|<V,|\zeta_Y(\alpha)|<V)$. Now by splitting the expectation over $\mathcal{A}$ and its complement, we get
		\begin{align}\label{eq1}
			\mathbb{E}\big[|\zeta_X(\alpha)\zeta_Y(\alpha)D(X,Y)|^2\big] \leq&V^4\|f\|_2^2+V^{2-2l}\mathbb{E}\big[|\zeta_Y(\alpha)|^{2+2l}|D(X,Y)|^2\big]\\
			&+V^{2-2l}\mathbb{E}\big[|\zeta_X(\alpha)|^{2+2l}|D(X,Y)|^2\big]\nonumber\\
			&+V^{-4l}\mathbb{E}\big[|\zeta_X(\alpha)|^{2+2l}|\zeta_Y(\alpha)|^{2+2l}|D(X,Y)|^2\big]. \nonumber
		\end{align}
		By Cauchy-Schwarz inequality, $|D(X,Y)|^2\leq \|f\|_2^2K$. Using the inequality above, we get:
		\begin{align}\label{eq2}
			\mathbb{E}\big[|\zeta_X(\alpha)\zeta_Y(\alpha)D(X,Y)|^2\big] \leq &\|f\|_2^2 \,(V^4 +V^{2-2l}K\mathbb{E}\big[|\zeta_Y(\alpha)|^{2+2l}\big]\\
			&+V^{2-2l}K\mathbb{E}\big[|\zeta_X(\alpha)|^{2+2l}\big]
			+V^{-4l}K\mathbb{E}\big[|\zeta_X(\alpha)|^{2+2l}|\zeta_Y(\alpha)|^{2+2l}\big]. \nonumber
		\end{align}
		\begin{case}{ $\alpha=1.$} Using Lemma~\ref{moment lemma} in (\ref{eq2}) we get the upper bound
			\begin{align*}
				\|f\|_2^2(V^4+V^{2-2l}K\exp(Cl\log\log l)).
			\end{align*}
			We choose $l=\log K+3$ and $V=(\log l)^C$, where $C$ is a large positive constant, and obtain
			\begin{align*}
				\mathbb{E}\big[|\zeta_X(\alpha)\zeta_Y(\alpha)D(X,Y)|^2\big] \leq \|f\|_2^2(\log\log K)^{O(1)}.
			\end{align*}
		\end{case}
		
		\begin{case}{ $\frac{1}{2} < \alpha <1$}.  Applying Lemma~\ref{moment lemma} in (\ref{eq2}) to get
			\begin{align*}
				\mathbb{E}\big[|\zeta_X(\alpha)\zeta_Y(\alpha)D(X,Y)|^2\big]\leq \|f\|_2^2\,(V^4+V^{2-2l}K\exp(C'(\alpha)l^{1/\alpha}(\log l)^{-1})).
			\end{align*}
			Choosing $l=(\log K)^\alpha(\log\log K)^\alpha+3$ and $V= \exp(C(\alpha)l^{-1+1/\alpha}(\log l)^{-1})$ we get the upper bound
			\begin{align*}
				\|f\|_2^2\, \exp{(C(\alpha)(\log K)^{1-\alpha}(\log\log K)^{-\alpha})}.
			\end{align*}
		\end{case}
		
		\begin{case}{ $\alpha =\frac{1}{2}$.} Let $\beta=\frac{1}{2}+\frac{1}{\log \|f\|_1}$. Then, (\ref{eq2}) and Lemma \ref{moment lemma} give us
			\begin{align}\label{eq4}
				\mathbb{E}\big[|\zeta_X(\beta)\zeta_Y(\beta)D(X,Y)|^2\big] \leq&\|f\|_2^2\big(V^4+V^{2-2l}K\exp(C'l^2\log\log\|f\|_1)\big).
			\end{align}
			
			By H\"{o}lder's inequality we get
			\begin{align}\label{eq3}
				S_f(2, 1/2)\leq \Big(S_f(2,\beta)\Big)^{\frac{1}{2\beta}}\|f\|_1^{2(1-\frac{1}{2\beta})}
				\ll \Big(S_f(2,\beta)\Big)^{\frac{1}{2\beta}}.
			\end{align}
			For $s \rightarrow 1$ we have \[\zeta(s)=\frac{1}{s-1}+O(1).\] Using this in \eqref{eq0} and combining with \eqref{eq4}) and \eqref{eq3}, we obtain the upper bound
			\begin{align}\label{eq5}
				S_f(2, 1/2) \ll	(\log\|f\|_1+O(1))^{-2}\|f\|_2^2\big(V^4+V^{2-2l}K\exp(C'l^2\log\log\|f\|_1)\big).
			\end{align}
			Hence we choose $l=(\log K)^{1/2}(\log\log\|f\|_1)^{-1/2}$ and $V=\exp(C''(l\log\log \|f\|_1))$ in \eqref{eq5} to get the upper bound
			\[S_f(2, 1/2) \ll (\log\|f\|_1+O(1))^{-2}\|f\|_2^2\exp(C(\log K\log\log \|f\|_1)^{1/2}).\]
		\end{case}
		\noindent This completes the proof.
	\end{proof}

	\section{Proof of Theorem \ref{infinity-PPC thm}}
	\begin{proof}
		Let $s>0$ be fixed and $N\geq (2s)^d$ be an integer. For $\boldsymbol{\alpha}\in \mathbb{R}^d $ and  $\boldsymbol{a}_n\in \mathbb{N}^d$, consider the sequence $(\boldsymbol{x_n})=(\{\boldsymbol{a}_n\boldsymbol{\alpha}\})$. Then,
		\begin{align*}
			R_{2,\infty}^{(d)}(s,\boldsymbol{\alpha},N)=\frac{1}{N}\displaystyle\sum_{1\leq m\neq n\leq N}\chi_{s,N}(\boldsymbol{\alpha}(\boldsymbol{a}_m-\boldsymbol{a}_n)),
		\end{align*}
		where $\chi_{s,N}$ is a characteristic function defined for $\boldsymbol{x}\in \mathbb{R}^d$ as follows:
		\begin{align*}
			\chi_{s,N}(\boldsymbol{x})=\begin{cases}
				1\quad& \text{if } \|\boldsymbol{x}\|_\infty\leq s/N^{1/d},\\
				0 \quad& \text{otherwise.}
			\end{cases}
		\end{align*}
		The Fourier series expansion of $\chi_{s,N}$ is given by
		\begin{align}\label{char}
			\chi_{s,N}(\boldsymbol{\alpha})\sim\displaystyle\sum_{\substack{\boldsymbol{r}\in\mathbb{Z}^d\\ \boldsymbol{r}=(r_1,\cdots,r_d)}}c_{\boldsymbol{r}}e(\boldsymbol{r}.\boldsymbol{\alpha}),
		\end{align}
		where
		\begin{align*}
			c_{\boldsymbol{r}}=c_{\boldsymbol{r},s}=&\displaystyle\int_{-s/N^{1/d}}^{s/N^{1/d}}\cdots\int_{-s/N^{1/d}}^{s/N^{1/d}}e{\Big(-\sum_{i\leq d}r_i\alpha_i\Big)}d\alpha_1\cdots d\alpha_d\nonumber\\
			=& c_{r_1}\cdots c_{r_d},
		\end{align*}
		and 
		\[c_{r_j}=\displaystyle\int_{-s/N^{1/d}}^{s/N^{1/d}}e{(-r_j\alpha_j)}d\alpha_j, \quad j=1,2,\ldots,d.
		\]
		
		Further, one gets the following upper bound:
		\begin{align*}
			|c_{r_j}|\leq\min{\Big(2s N^{-\frac{1}{d}},\, |r_j|^{-1}\Big)}.
		\end{align*}
		A straightforward calculation gives the expectation:
		\begin{align*}
			\mathbb{E}\big[R_{2,\infty}^{(d)}(s,.,N)\big]=\displaystyle\int_{[0,1)^d}R_{2,\infty}^{(d)}(s,\boldsymbol{\alpha},N)d\boldsymbol{\alpha}=(2s)^d\frac{N-1}{N}.
		\end{align*}
		Now, the variance of $R_{2,\infty}^{(d)}(s,\boldsymbol{\alpha},N) $ is defined as 
		\begin{align*}
			\Var(R_{2,\infty}^{(d)}(s,.,N)) :=\int_{[0,1)^d}\bigg(R_{2,\infty}^{(d)}(s,\boldsymbol{\alpha},N)-\frac{(2s)^d(N-1)}{N}\bigg)^2d\boldsymbol{\alpha}.
		\end{align*}
		By using Fourier series expansion of $\chi_{s,N}$ from \eqref{char}, we write
		\begin{align*}
			\Var(R_{2,\infty}^{(d)}(s,.,N)) = \frac{1}{N^2}\displaystyle\int_{[0,1)^d}\Bigg(\sum_{1\leq m\neq n\leq N}\sum_{\substack{\boldsymbol{r}\in\mathbb{Z}^d\setminus \{\boldsymbol{0}\}}}c_{\boldsymbol{r}}e(\boldsymbol{r}.(\boldsymbol{\alpha}(\boldsymbol{a}_m-\boldsymbol{a}_n)))\Bigg)^2d\boldsymbol{\alpha}.
		\end{align*}
		After squaring the integrand, we want to interchange the summations and integrations. Such rearrangement can be justified by using the fact that the partial sums of a Fourier series of an indicator function are uniformly bounded. Hence, the dominated convergence theorem is applicable ( for example, see,  \cite[Chapter 3, Exercise 18]{SS}).
		Thus, $\Var(R_{2,\infty}^{(d)}(s,.,N))$ equals
		\begin{align*}
			&\frac{1}{N^2}\displaystyle\sum_{\substack{1\leq m\neq n\leq N\\1\leq k\neq l\leq N}}\sum_{\substack{\boldsymbol{r},\boldsymbol{t}\in\mathbb{Z}^d\setminus \{\boldsymbol{0}\}}}c_{\boldsymbol{r}}c_{\boldsymbol{t}}\int_{[0,1)^d}e(\boldsymbol{r}.(\boldsymbol{\alpha}(\boldsymbol{a}_m-\boldsymbol{a}_n))-\boldsymbol{t}.(\boldsymbol{\alpha}(\boldsymbol{a}_k-\boldsymbol{a}_l)))d\boldsymbol{\alpha}\\
			=& \frac{1}{N^2}\displaystyle\sum_{\substack{1\leq m\neq n\leq N\\1\leq k\neq l\leq N}}\sum_{\substack{\boldsymbol{r},\boldsymbol{t}\in\mathbb{Z}^d\setminus \{\boldsymbol{0}\}}}c_{\boldsymbol{r}}c_{\boldsymbol{t}}\prod_{i\leq d}\int_{[0,1)}e\Big(\big(r_i(a_m^{(i)}-a_n^{(i)})-t_i(a_k^{(i)}-a_l^{(i)})\big)\alpha_i\Big)d\alpha_i.
		\end{align*}
		But, the innermost integral is equal to $1$ whenever $r_i=t_i=0 $ or  $r_i(a_m^{(i)}-a_n^{(i)})=t_i(a_k^{(i)}-a_l^{(i)})$, otherwise the integral is zero.
		Note that when all components of $\boldsymbol{r}, \boldsymbol{t}$ are nonzero, the associated sums will produce the main contribution. Otherwise, when some of them are zero, we save powers of $N$. 
		Precisely, the variance is equal to
		\begin{align}\label{eq6}
			\frac{1}{N^2}\displaystyle\sum_{\substack{j_1<\cdots<j_q \leq d\\1\leq q\leq d}}\sideset{}{'}\sum_{\substack{\boldsymbol{v}=(v_{j_1},\ldots,v_{j_q}),\\\boldsymbol{w}=(w_{j_1},\dots,w_{j_q})\in\mathbb{Z}^q}}\mathcal{R}_N(v_{j_1},\ldots,v_{j_q})\mathcal{R}_N(w_{j_1},\ldots,w_{j_q})\Big(\frac{2s}{N^{1/d}}\Big)^{2(d-q)}\sideset{}{'}\sum_{\substack{\boldsymbol{r},\boldsymbol{t}\in\mathbb{Z}^q\\ \boldsymbol{r}\boldsymbol{v}=\boldsymbol{t}\boldsymbol{w}}}c_{\boldsymbol{r}} c_{\boldsymbol{t}},
		\end{align}
		where $\mathcal{R}_N$ is as defined in \eqref{R_N}.
		
		The relation $r_iv_i=t_iw_i$ allows us to write $r_i, t_i$ as
		\begin{align}
			r_i=\frac{h_iw_i}{\gcd(v_i,w_i)} \text{ and } t_i=\frac{h_iv_i}{\gcd(v_i,w_i)},\nonumber
		\end{align}
		where $h_i$ is a nonzero integer. Now we separate the sum over $h_i$ into three sub-intervals of the real line given by
		\begin{align}
			&|h_i|\leq\frac{N^{1/d}\gcd(v_i,w_i)}{s\max(|v_i|,|w_i|)},\nonumber\\
			&\frac{N^{1/d}\gcd(v_i,w_i)}{s\max(|v_i|,|w_i|)}<|h_i|\leq\frac{N^{1/d}\gcd(v_i,w_i)}{s\min(|v_i|,|w_i|)},\nonumber\\
			&|h_i|>\frac{N^{1/d}\gcd(v_i,w_i)}{s\min(|v_i|,|w_i|)}.\nonumber
		\end{align}
		For any fixed $j_1<j_2<\cdots<j_q$, we follow the arguments in page 344 of ~\cite{hinrichs2019multi} and get
		\[\sideset{}{'}\sum_{\substack{\boldsymbol{r},\,\boldsymbol{t}\in\mathbb{Z}^q\\ \boldsymbol{r}\boldsymbol{v}=\boldsymbol{t}\boldsymbol{w}}}c_{\boldsymbol{r}} c_{\boldsymbol{t}}\ll\frac{s^q(\log N)^q}{N^{q/d}}\displaystyle\prod_{1\leq \beta\leq q}\frac{\gcd(v_{j_\beta},w_{j_\beta})}{\sqrt{v_{j_\beta}w_{j_\beta}}}.
		\]
		
		Replace the above bound in the right hand side of \eqref{eq6} to get
		\begin{align}
			\Var(R_2^{(d)}(s,.,N))&\ll\frac{s^d(\log N)^d}{N^3}\displaystyle\sideset{}{'}\sum_{\boldsymbol{v},\boldsymbol{w}\in\mathbb{Z}^d}\mathcal{R}_N(\boldsymbol{v})\mathcal{R}_N(\boldsymbol{w})\prod_{1\leq i\leq d}\frac{\gcd(v_i,w_i)}{\sqrt{v_iw_i}}\nonumber\\
			+&\sum_{\substack{j_1<\cdots<j_q\\1\leq q<d}}\frac{s^{2d-q}(\log N)^q}{N^{4-q/d}}\sideset{}{'}\sum_{\boldsymbol{v},\boldsymbol{w}\in\mathbb{Z}^q}\mathcal{R}_N(\boldsymbol{v})\mathcal{R}_N(\boldsymbol{w}))\prod_{1\leq \beta\leq q}\frac{\gcd(v_{j_\beta},w_{j_\beta})}{\sqrt{v_{j_\beta}w_{j_\beta}}}. \nonumber
		\end{align}
		
		Now applying Proposition~\ref{gen gcd lemma} to the above GCD sums, we get for any $\gamma>0$,
		\begin{align*}
			\Var(R_2^{(d)}(s,.,N)) &\ll \frac{s^d(\log N)^d}{N^3}E\left(A_N^{(1)};\ldots;A_N^{(d)}\right)\frac{\exp(C\sqrt{\log N\log\log N})}{(\log N+\O(1))^d}+\frac{1}{N^{1/d-\gamma}}\\
			&\ll  \frac{s^d}{N^3}E\left(A_N^{(1)};\ldots;A_N^{(d)}\right)\exp\Big(C\sqrt{\log N\log\log N}\Big).
		\end{align*}
		Thus, under the hypothesis of the theorem, it follows that the variance is $\O(N^{-\delta/2})$. This is the main part of the proof. The rest of the arguments follow a standard method of applying Chebyshev's inequality and Borel-cantelli lemma(see the proof of Theorem 1 in \cite{aistleitner2017additive} and also \cite{hinrichs2019multi, RZ1999}.)
	\end{proof}

	\section{Proof of theorems \ref{2-PPC thm1} and \ref{2-PPC thm2}.}
	\subsection{Properties of Bessel functions}
	Here we state some properties of the Bessel function, which are important tools in the proof of Theorem \ref{2-PPC thm1} and Theorem~\ref{2-PPC thm2}. For $\nu$ complex order with $\Re(\nu)>-1/2$ and $t\geq 0$ the Bessel function $J_\nu$ is defined by 
	\begin{align}\label{bessel1}
		J_\nu(t) = \frac{(t/2)^\nu}{\Gamma(\nu+1/2)\Gamma(1/2)}\int_{-1}^1e^{itx}(1-x^2)^{\nu}\frac{dx}{\sqrt{1-x^2}}.
	\end{align}
	This definition is also valid for $t\in\mathbb{C}.$ The function $J_\nu$ has many interesting properties. For our application, we mention a few of them below and these are essentially given in \cite[Appendix B]{loukasgrafakos}.
	For $\Re(\nu)>-1/2$ and $t\geq 1$, we have a nice approximation that gives an asymptotic formula as $t\rightarrow \infty,$
	\begin{align}\label{p1}
		J_\nu(t) = \sqrt{\frac{2}{\pi t}} \cos{\left(t-\frac{\pi \nu}{2}-\frac{\pi}{4}\right)} +\O_\nu(t^{-3/2}).
	\end{align}
	For $0< t\leq 1$ and $\Re(\nu)>-1/2$, we have the following upper bound
	\begin{align}\label{p2}
		J_\nu(t) \ll_\nu \exp\big({\max\{(\Re(\nu)+1/2)^{-2}, (\Re(\nu)+1/2)^{-1} \}|\Im(\nu)|^2}\big) t^{\Re(\nu)}.
	\end{align}
	Whenever $t>0$ and $\nu \in \mathbb{N}$, the Bessel function reduces to a simple form:
	\begin{align}\label{bessel2}
		J_\nu(t) = \frac{1}{2\pi}\int_0^{2\pi}\cos(t\sin{\theta}- \nu \theta) d\theta.
	\end{align}
	
	\begin{lemma}\label{bessel fn lemma}
		Let $N$, $\nu \geq 2$ be two integers and $s,r>0$ be real numbers. Then
		\begin{itemize}
			\item[(1)] $J_{\nu/2}(r)\ll_\nu 1$ if $r\leq 1,$\\
			\item[(2)] $J_{\nu/2}(r)\ll_\nu \frac{1}{\sqrt{r}}$ if $r > 1,$\\
			\item[(3)] $\frac{J_{\nu/2}( 2\pi s r N^{-{1}/{\nu}})}{r^{\nu/2}}\ll_\nu \frac{ s^{\nu/2}}{\sqrt{N}}$, for any $r>0,$\\
			\item[(4)] $|J_\mu(r)|\leq 1$, for any $r>0$, and $\mu\in\mathbb{N}.$
		\end{itemize}	
	\end{lemma}
	Properties $1,\,2,\,3$ and $4$ follows from \eqref{p2},\,\eqref{p1},\, \eqref{bessel1} and \eqref{bessel2}, respectively.
	\subsection{Preparation of the proofs of Theorem \ref{2-PPC thm1} and Theorem \ref{2-PPC thm2}}
	For the sequence $(\boldsymbol{x}_n)= (\{\boldsymbol{a}_n\boldsymbol{\alpha}\})$, the $d$-dimensional pair correlation statistic in $2$-norm is
	\begin{align}
		R_{2,2}^{(d)}(s,\boldsymbol{\alpha},N)=\frac{1}{N}\displaystyle\sum_{1\leq m\neq n\leq N}I_{s,N}(\boldsymbol{\alpha}(\boldsymbol{a}_m-\boldsymbol{a}_n)),\nonumber
	\end{align}
	where $I_{s,N}$ is the indicator function for all $\boldsymbol{x}\in\mathbb{R}^d$  which satisfy $\|\boldsymbol{x}\|_2\leq s/N^{1/d}$.
	Then, the Fourier series expansion of $I_{s,N}$ is given by
	\begin{align}
		I_{s,N}(\boldsymbol{\alpha})\sim\displaystyle\sum_{\substack{\boldsymbol{r}\in\mathbb{Z}^d\\ \boldsymbol{r}=(r_1,\cdots,r_d)}}c_{\boldsymbol{r}}e(\boldsymbol{r}.\boldsymbol{\alpha}),\nonumber
	\end{align}
	where
	\begin{align}\label{c_r}
		c_{\boldsymbol{r}}=c_{\boldsymbol{r},s}=\begin{cases}
			\omega_d\frac{s^d}{N}, & \mbox{ if } \boldsymbol{r}=\boldsymbol{0},\\
			\frac{s^{d/2}}{\sqrt{N}\|\boldsymbol{r}\|_2^{d/2}}J_{d/2}\big(\frac{2\pi s}{N^{1/d}}\|\boldsymbol{r}\|_2\big), & \mbox{ if } \boldsymbol{r}\in\mathbb{Z}^d\setminus \{\boldsymbol{0}\}.
		\end{cases}
	\end{align} 
	For the details of such formulas, one can see~\cite[Appendix B]{loukasgrafakos}. 
	
	It is not hard to verify that  
	\begin{align*}
		\mathbb{E}(R_{2,2}^{(d)}(s, . ,N))=\omega_ds^d\frac{N-1}{N}.
	\end{align*}
	We now proceed to calculate the variance of $R_{2,2}^{(d)}(s, . ,N)$, that is,
	\begin{align*}
		\Var(R_{2,2}^{(d)}(s,.,N)) &= \displaystyle\int_{[0,1)^d}\bigg(R_{2,2}^{(d)}(s,\boldsymbol{\alpha},N)-\frac{\omega_ds^d(N-1)}{N}\bigg)^2d\boldsymbol{\alpha}\\
		& = \frac{1}{N^2}\displaystyle\int_{[0,1)^d}\Bigg(\sum_{1\leq m\neq n\leq N}\sum_{\substack{\boldsymbol{r}\in\mathbb{Z}^d\setminus \{\boldsymbol{0}\}}}c_{\boldsymbol{r}}e(\boldsymbol{r}.(\boldsymbol{\alpha}(\boldsymbol{a}_m-\boldsymbol{a}_n)))\Bigg)^2d\boldsymbol{\alpha}.
	\end{align*}
	By interchanging summation and integration $\Var(R_{2,2}^{(d)}(s,.,N))$ is seen to be bounded above by
	\begin{align}\label{eqppc2}
		\frac{1}{N^2}\displaystyle\sum_{\substack{1\leq m\neq n\leq N\\1\leq k\neq l\leq N}}\sum_{\substack{\boldsymbol{r},\boldsymbol{t}\in\mathbb{Z}^d\setminus \{\boldsymbol{0}\} }}c_{\boldsymbol{r}}\overline{c_{\boldsymbol{t}}}\prod_{i\leq d}\int_{[0,1)}e\Big(\big(r_i(a_m^{(i)}-a_n^{(i)})-t_i(a_k^{(i)}-a_l^{(i)})\big)\alpha_i\Big)d\alpha_i.
	\end{align}
	Such rearrangement can be justified as follows. We split the Fourier series of the indicator function at some parameter $M$ and consider the partial sum up to $M$ and the tail part. We apply the Cauchy-Schwarz inequality to the tail part to interchange the sums and the integration, which leaves us with the square-integral of the tail. However, the $2$-norm of the tail goes to zero as $M\rightarrow \infty$. Thus, for arbitrarily large $M$ the partial sum up to $M$ remains, and the rearrangement is obvious.
	
	As in Theorem~\ref{infinity-PPC thm}, it is enough to show that the variance is arbitrarily small.
	\begin{proof}[Proof of Theorem~\ref{2-PPC thm1}]
		In this theorem, we consider $a_n^{(i)}=a_n$ for all $1\leq i\leq d,$ then \eqref{eqppc2} and the orthogonality of exponential function allow us to write
		\begin{align}	
			\Var(R_{2,2}^{(d)}(s,.,N))\leq&\frac{1}{N^2}\displaystyle\sum_{v,w\in\mathbb{Z}\setminus \{{0}\}}\mathcal{R}_N(v)\mathcal{R}_N(w)\sum_{\substack{\boldsymbol{r},\boldsymbol{t}\in\mathbb{Z}^d\setminus \{\boldsymbol{0}\}\\ vr_i=wt_i}}|c_{\boldsymbol{r}}c_{\boldsymbol{t}}|. \label{eq1ppc2}
		\end{align}
		Let us fix $j_1<j_2<\cdots<j_q\leq d$ where $1\leq q\leq d$. We claim that  
		\begin{align}\label{eq3ppc2}
			\sideset{}{''}\sum_{\substack{\boldsymbol{r},\boldsymbol{t}\in\mathbb{Z}^d\setminus \{\boldsymbol{0}\}\\ r_{j_\beta}v=r_{j_\beta}w}}|c_{\boldsymbol{r}}c_{\boldsymbol{t}}|\ll_d\frac{s^{2d-q}}{N^{2-q/d}}\frac{(v,w)^q}{\sqrt{|vw|^q}},
		\end{align}
		where the sum $\Sigma^{''}$ is over $\boldsymbol{r},\boldsymbol{t}$ with all components zero other than $r_{j_\beta}, t_{j_\beta}$ for $1\leq\beta\leq q.$ This implies that for $q<d$, we will always save a power of $N$ and the main contribution to the variance comes from the case $q=d$. 
		
		Now, from the relations $r_{j_\beta}v=r_{j_\beta}w$ we can write
		\[r_{j_\beta}=\frac{wh_{j_\beta}}{(v,w)}\text{ and }t_{j_\beta}=\frac{vh_{j_\beta}}{(v,w)}\text{ for }1\leq\beta\leq q,\]
		where $h_{j_\beta}$ takes integer values. Let us denote the $q$-tuple $(h_{j_1},\ldots,h_{j_q})$ by $\boldsymbol{h}$ and from now onwards, we use the notation $\|\boldsymbol{h}\|_2$  for the usual $2$-norm.
		
		Using the definition of $c_{\boldsymbol{r}}$'s from \eqref{c_r} the left hand side of \eqref{eq3ppc2} is bounded by
		\begin{align}\label{2nd eq}
			\ll \displaystyle\frac{s^d}{N}\sideset{}{'}\sum_{\boldsymbol{h}\in\mathbb{Z}^{q}}\frac{(v,w)^d}{(|vw|)^{d/2}}\frac{|J_{d/2}(\frac{2\pi s|w|}{N^{1/d}(v,w)}\|\boldsymbol{h}\|_2)|}{\|\boldsymbol{h}\|_2^{d/2}}\frac{|J_{d/2}(\frac{2\pi s|v|}{N^{1/d}(v,w)}\|\boldsymbol{h}\|_2)|}{\|\boldsymbol{h}\|_2^{d/2}}.
		\end{align}
		Now we divide this sum into three parts 
		$\Sigma^{'}=\Sigma_1^{'}+\Sigma_2^{'}+\Sigma_3^{'}$ depending on the following range of values of $\|\boldsymbol{h}\|_2:$
		\begin{align}
			1\leq&\|\boldsymbol{h}\|_2\leq\frac{N^{1/d}(v,w)}{2\pi s\max(|v|,|w|)}:=\max \boldsymbol{h},\nonumber \\
			\frac{N^{1/d}(v,w)}{2\pi s\max(|v|,|w|)}\leq&\|\boldsymbol{h}\|_2\leq\frac{N^{1/d}(v,w)}{2\pi s\min(|v|,|w|)}:=\min \boldsymbol{h},\nonumber\\
			&\|\boldsymbol{h}\|_2\geq\frac{N^{1/d}(v,w)}{2\pi s\min(|v|,|w|)}.\nonumber
		\end{align}
		We are using (3) of Lemma~\ref{bessel fn lemma} to get
		\begin{align*}
			\Sigma_1^{'} &\ll \frac{s^d}{N}(\max \boldsymbol{h})^q\ll\frac{s^{d-q}(v,w)^q}{N^{1-q/d}\sqrt{|vw|^q}}.
		\end{align*}
		By using (1) and (2) of Lemma~\ref{bessel fn lemma}, we obtain
		\begin{align}
			\Sigma_2^{'}&\ll\displaystyle\frac{(v,w)^d}{(|vw|)^{d/2}}\sum_{\max \boldsymbol{h}\leq\|\boldsymbol{h}\|_2\leq \min \boldsymbol{h}}\frac{\sqrt{\max \boldsymbol{h}}}{\|\boldsymbol{h}\|_2^{d+1/2}}\nonumber\\
			&\ll\frac{s^{d-q}(v,w)^q\max(|v|,|w|)^{d-q}}{N^{1-q/d}|vw|^{d/2}}\ll\frac{s^{d-q}(v,w)^q}{N^{1-q/d}\sqrt{|vw|^q}}.\nonumber
		\end{align}
		Again, using case (2) of Lemma~\ref{bessel fn lemma}, we deduce
		\begin{align*}
			\Sigma_3^{'}&\ll \frac{(v,w)^{d+1}N^{1/d}}{s(|vw|)^{d/2+1/2}}\displaystyle\sum_{\|\boldsymbol{h}\|_2> \min \boldsymbol{h}}\frac{1}{\|\boldsymbol{h}\|_2^{d+1}}\\
			&\ll \frac{(v,w)^{d+1}N^{1/d}}{s(|vw|)^{d/2+1/2}}\frac{1}{(\min \boldsymbol{h})^{d-q+1}} \ll \frac{s^{d-q}(v,w)^q}{N^{1-q/d}\sqrt{|vw|^q}}.
		\end{align*}
		Combining all three bounds with \eqref{2nd eq} we obtain the claim \eqref{eq3ppc2}. By inserting \eqref{eq3ppc2} into \eqref{eq1ppc2}, we get
		\begin{align*}
			\Var(R_{2,2}^{(d)}(s,.,N))\ll &\displaystyle\sum_{1\leq q\leq d}\frac{s^{2d-q}}{N^{4-q/d}}\sum_{v,w\in\mathbb{Z}\setminus \{0\}}\mathcal{R}_N(v)\mathcal{R}_N(w)\frac{(v,w)^q}{\sqrt{|vw|^q}}.\\
		\end{align*}
		Now, using the result on GCD sums by G\'al~\cite{gel}, we obtain
		\begin{align*}
			\Var(R_{2,2}^{(d)}(s,.,N)) \ll \begin{cases}
				\frac{s^2}{N^3}E(A_N)(\log\log N)^2+ \frac{s^2}{N^{1/2 -\gamma }},& \mbox{ if } d=2,\\
				\frac{s^d}{N^3}E(A_N) + \frac{s^d}{N^{{1}/{d}- \gamma}}, & \mbox{ if } d\geq 3,
			\end{cases}
		\end{align*}
		for some sufficiently small $\gamma>0.$ Thus, under the hypothesis of this theorem, the variance is $\O((\log N)^{-1-\epsilon/2})$. The rest of the arguments follow from the proof of ~\cite[Theorem 5]{Bloom2019GCDSA}.
	\end{proof}
	\begin{proof}[Proof of Theorem \ref{2-PPC thm2}]
		For $d=2$, (\ref{eqppc2}) gives us
		\begin{align}\label{eq4ppc2}
			\Var(R_{2,2}^{(2)}(s,.,N))&\leq \frac{1}{N^2}\sum_{\substack{v_1,v_2,w_1,w_2\in\mathbb{Z}\setminus \{0\}}}\mathcal{R}_N(v_1,v_2)\mathcal{R}_N(w_1,w_2)\sideset{}{'}\sum_{\substack{\boldsymbol{r}, \boldsymbol{t}\in\mathbb{Z}^2\\ r_iv_i=t_iw_i}}|c_{\boldsymbol{r}}c_{\boldsymbol{t}}| \\ \nonumber
			& + \frac{1}{N^2}\sum_{\substack{v_1, w_1\in\mathbb{Z}\setminus \{0\}}}\mathcal{R}_N(v_1)\mathcal{R}_N(w_1) \sum_{\substack{r_1,t_1 \in\mathbb{Z}\setminus \{0\}\\ r_1v_1=t_1w_1}}|c_{(r_1,0)}c_{(t_1,0)}|\\ \nonumber
			& + \frac{1}{N^2}\sum_{\substack{v_2, w_2\in\mathbb{Z}\setminus \{0\}}}\mathcal{R}_N(v_2)\mathcal{R}_N(w_2)  \sum_{\substack{r_2,t_2\in\mathbb{Z}\setminus \{0\}\\ r_2v_2=t_2w_2}}|c_{(0,r_2)}c_{(0,t_2)}|.
		\end{align} 
		Note that the second and third terms on the right hand side above are $\O(N^{-1/2+\gamma})$ for some $\gamma>0$, as it follows from \eqref{eq3ppc2} with $d=2$ and $q=1$. 
		
		Let us call the inner sum in the first term $I$.
		From the relation $r_iv_i=t_iw_i$ we have 
		$r_i=\frac{w_ih_i}{(v_i,w_i)}$ and $t_i=\frac{v_ih_i}{(v_i,w_i)}$ for $i=1,2$ where $h_i$'s are nonzero integers. For simplicity we write
		\[A=\frac{w_1}{(v_1,w_1)},\, B=\frac{w_2}{(v_2,w_2)},\, C=\frac{v_1}{(v_1,w_1)},\, D=\frac{v_2}{(v_2,w_2)}.\]
		By using \eqref{c_r} we obtain
		\begin{align*}
			I&\ll\frac{s^2}{N}\sum_{h_1,h_2\neq0}\frac{\Big|J_1\Big(\frac{2\pi s}{\sqrt{N}}(A^2h_1^2+B^2h_2^2)^{\frac{1}{2}}\Big)\Big|}{(A^2h_1^2+B^2h_2^2)^{\frac{1}{2}}}\frac{\Big|J_1\Big(\frac{2\pi s}{\sqrt{N}}(C^2h_1^2+D^2h_2^2)^{\frac{1}{2}}\Big)\Big|}{(C^2h_1^2+D^2h_2^2)^{\frac{1}{2}}}\nonumber\\
			= &\frac{s^2}{N|AC|}\sum_{h_1\neq 0}\frac{1}{h_1^2}\sum_{h_2\neq 0}\frac{\Big|J_1\Big(\frac{2\pi s|Ah_1|}{\sqrt{N}}(1+\frac{B^2}{A^2h_1^2}h_2^2)^{\frac{1}{2}}\Big) J_1\Big(\frac{2\pi s|Ch_1|}{\sqrt{N}}(1+\frac{D^2}{C^2h_1^2}h_2^2)^{\frac{1}{2}}\Big)\Big|}{\Big(1+\frac{B^2}{A^2h_1^2}h_2^2\Big)^{\frac{1}{2}} \Big(1+\frac{D^2}{C^2h_1^2}h_2^2\Big)^{\frac{1}{2}}}.
		\end{align*}
		Now we divide the sum over $h_1$ into two parts,
		\[|h_1|\leq\frac{\sqrt{N}}{2\pi s\min(|A|,|C|)}\text{ and }|h_1|>\frac{\sqrt{N}}{2\pi s\min(|A|,|C|)}.\]
		Applying (4) and (2) of Lemma \ref{bessel fn lemma}, respectively, for the small and large arguments of the Bessel function in the above inequality, we get 
		\begin{align*}
			I	&\ll \frac{s^2}{N|AC|}\Bigg(\sum_{|h_1|\leq\frac{\sqrt{N}}{2\pi s\min(|A|,|C|)}}\frac{1}{h_1^2}\sum_{h_2\neq 0}\Big(1+\frac{B^2}{A^2h_1^2}h_2^2\Big)^{-\frac{1}{2}}\Big(1+\frac{D^2}{C^2h_1^2}h_2^2\Big)^{-\frac{1}{2}}\\
			&+\sum_{|h_1|>\frac{\sqrt{N}}{2\pi s\min(|A|,|C|)}}\frac{\sqrt{N}}{s|AC|^{1/2}|h_1|^3}\sum_{h_2\neq 0}\Big(1+\frac{B^2}{A^2h_1^2}h_2^2\Big)^{-\frac{1}{2}-\frac{1}{4}}\Big(1+\frac{D^2}{C^2h_1^2}h_2^2\Big)^{-\frac{1}{2}-\frac{1}{4}}\Bigg).
		\end{align*}
		Now observe that for $a_1,a_2\in\mathbb{R}$, $(1+a_1^2)(1+a_2^2)\geq(1+a_1a_2)^2$. Then, the sum on $h_2$ in the above two terms can be bounded by
		\[\int_0^{\infty}\Big(1+\Big|\frac{BD}{AC}\Big|\frac{x^2}{h_1^2}\Big)^{-1}dx\,
		\text{ and } \,\int_0^{\infty}\Big(1+\Big|\frac{BD}{AC}\Big|\frac{x^2}{h_1^2}\Big)^{-3/2}dx,\]
		respectively, and both are bounded by
		$|h_1|\big|\frac{AC}{BD}\big|^{1/2}.$
		Thus, we have 
		\begin{align*}
			I & \ll \frac{s^2}{N|AC|}\Bigg(\Big|\frac{AC}{BD}\Big|^{1/2}\sum_{|h_1|\leq\frac{\sqrt{N}}{2\pi s\min(|A|,|C|)}}\frac{1}{|h_1|}
			+ \frac{\sqrt{N}}{s \sqrt{|BD|}}\sum_{|h_1|>\frac{\sqrt{N}}{2\pi s\min(|A|,|C|)}}\frac{1}{|h_1|^2}\Bigg)\\
			& \ll\frac{s^2}{N|AC|}\Bigg(\Big|\frac{AC}{BD}\Big|^{1/2}\log N +\frac{\min(|A|,|C|)}{|BD|^{1/2}}\Bigg)\ll \frac{s^2}{N}\frac{\log{N}}{|ABCD|^{1/2}}.
		\end{align*}
		
		Hence, we deduce the variance estimate 
		\begin{align}\label{estimate1}
			\Var(R_{2,2}^{(2)}(s,.,N))\ll \frac{s^2\log N}{N^3}\displaystyle\sum_{\substack{v_1,v_2,\\w_1,w_2\in\mathbb{Z}\setminus \{0\}}}\mathcal{R}_N(v_1,v_2)\mathcal{R}_N(w_1,w_2)\frac{(v_1,w_1)(v_2,w_2)}{\sqrt{|v_1w_1v_2w_2|}}.
		\end{align}
		Now applying Proposition~\ref{gen gcd lemma} in \eqref{estimate1}, we conclude that 
		\[\Var(R_{2,2}^{(2)}(s,.,N))\ll \frac{s^2}{N^3\log{N}}E(A_N^{(1)};A_N^{(2)})\exp(C\sqrt{\log N\log\log N}).\]
		This completes the proof.
	\end{proof}
	
	\section{Proof of Theorem~\ref{thm6}}
	This follows easily by combining Theorem \ref{infinity-PPC thm} with Corollary~\ref{coro2.3} and Theorem \ref{2-PPC thm2} with Corollary~\ref{coro2.3} .
	\section{Acknowledgement}  The authors would like to thank Prof. Christoph Aistleitner for helpful suggestions.


\begin{thebibliography}{99}
		
		
		
		\bibitem{aistleitner2017additive}
		C. Aistleitner, G. Larcher, M. Lewko, 
		\newblock Additive energy and the Hausdorff dimension of the exceptional set in metric pair correlation problems, with an appendix by J. Bourgain,
		\newblock{\em{Israel J. Math.}} 222 (2017), no. 1, 463-485.
		
		\bibitem{aistleitner2018pair}
		C. Aistleitner, T. Lachmann, F. Pausinger,
		\newblock Pair correlations and equidistribution,
		\newblock{\em{J. Number Theory}} 182 (2018), 206-220.
		
		\bibitem{Bloom2019GCDSA}
		T. F. Bloom, A. Walker,
		\newblock GCD sums and sum-product estimates,
		\newblock {\em{Israel J. Math.}} 235 (2019), 1-11.
		
		\bibitem{BZ} F. Boca, A. Zaharescu, 
		\newblock On the pair correlation for fractional parts of vector sequences,
		\newblock{\em{Arch. Math.}} 77 (2001), 498-507.		
		
		\bibitem{BS}
		A. Bondarenko, K. Seip,
		\newblock Large greatest common divisor sums and extreme values of the Riemann zeta function,
		\newblock{\emph{Duke Math. J.}} 166 (2017), 1685-1701.
		
		\bibitem{BT}
		R. de la Bret\`eche, G. Tenenbaum, 
		\newblock Sommes de G\'al et applications, 
		\newblock{\em{Proc. Lond. Math. Soc.}} 119 (2019), 104-134.
		
		\bibitem{FW}
		K. Ford, T. D. Wooley,
		\newblock On Vinogradov's mean value theorem: strongly diagonal behaviour via efficient congruencing,
		\newblock{\em{Acta Math.}} 213 (2014),no. 2, 199-236. 
		
		\bibitem{gel}
		I. S. G\'al, \newblock A theorem concerning diophantine approximations,
		\newblock{\em{Nieuw Arch. Wiskd.}} 2 (23) (1949), 13-38.
		
		\bibitem{Garaev2004}
		M. Z. Garaev,
		\newblock Upper bounds for the number of solutions of a diophantine equation,
		\newblock{\em{Trans. Amer. Math. Soc.}} 357 (2005), 2527-2534
		
		\bibitem{Garaev2001}
		M. Z. Garaev, Ka-Lam Kueh,
		\newblock $L_1$-norms of exponential sums and the corresponding additive problem,
		\newblock {\em{Z. Anal. Anwendungen}}, 20(2001), No. 4, pp. 999-1006.
		
		\bibitem{loukasgrafakos}
		L. Grafakos, 
		\newblock{\em{Classical Fourier Analysis,}}
		\newblock Third edition, Graduate Texts in Mathematics, 249. Springer, New York, 2014.
		
		\bibitem{Grepstad2017Larcher}
		S. Grepstad, G. Larcher,
		\newblock On pair correlation and discrepancy,
		\newblock{\em{ Arch. Math.}} 109 (2017), 143-149. 
		
		
		\bibitem{hauke2021weak}
		M. Hauke, A. Zafeiropoulos, 
		\newblock Weak Poissonian correlations,
		\newblock preprint, 2021, arXiv:2112.11813.
		
		
		\bibitem{hinrichs2019multi}
		A. Hinrichs, L. Kaltenb{\"o}ck, G. Larcher, W. Stockinger, M. Ullrich, 
		\newblock On a multi-dimensional Poissonian pair correlation concept and uniform distribution,
		\newblock{\em{Monatsh. Math.}} 190 (2019), no. 2, 333-352.
		
		\bibitem{haltonnotPPC2021}
		R. Hofer, L. Kaltenböck,  
		\newblock Pair correlations of Halton and Niederreiter sequences are not Poissonian,
		\newblock Monatsh Math 194, 789–809 (2021).
		
		\bibitem{K}
		A. Klenke,
		\newblock {\em{Probability Theory: A Comprehensive Course,}}
		\newblock Springer Publishing, Cham, 2020.
		
		\bibitem{larcher2020pair}
		G. Larcher, W. Stockinger,
		\newblock On pair correlation of sequences, in Discrepancy theory, Radon Ser. Comput. Appl. Math., 26, De Gruyter, Berlin, 2020, pp. 133-145.
		
		\bibitem{LS}
		G. Larcher, W. Stockinger,
		\newblock Pair correlation of sequences $(\{a_n \alpha \})_{n\in N}$ with maximal additive energy,
		\newblock {\em{Math. Proc. Cambridge Philos. Soc.}} 168 (2020), no. 2, 287--293.
		
		
		\bibitem{lewko2017refinements}
		M. Lewko, M. Radziwi{\l}{\l},
		\newblock Refinements of G{\'a}l's theorem and applications,
		\newblock {\em{Adv. in Math.}} 305 (2017), 280-297.
		
		\bibitem{marklof2020pair}
		J. Marklof, 
		\newblock Pair correlation and equidistribution on manifolds,
		\newblock{\em{Monatsh. Math.}} 191 (2020), 2, 279-294.
		
		\bibitem{rudnick1998pair}
		Z. Rudnick, P. Sarnak, 
		\newblock The pair correlation function of fractional parts of polynomials,
		\newblock{\em{Comm. Math. Phys.}} 194 (1998), no. 1, 61-70.
		
		\bibitem{RSZ2001}
		Z. Rudnick, P. Sarnak, A. Zaharescu,
		\newblock The distribution of spacings between the fractional parts of
		$n^2\alpha$,
		\newblock{\em{Invent. Math.}} 145 (2001), no. 1, 37--57.
		
		\bibitem{RZ1999}
		Z. Rudnick, A. Zaharescu,
		\newblock A metric result on the pair correlation of fractional parts of sequences,
		\newblock{\em{Acta Arith.}} 89 (1999), 283--293.
		
		\bibitem{SS}
		E. M. Stein, R. Shakarchi,
		\newblock {\em{Fourier Analysis. An Introduction,}}
		\newblock  Princeton Lectures in Analysis, 1. Princeton University Press, Princeton, NJ, 2003. xvi+311 pp.
		
		\bibitem{steinerberger2020poissonian}
		S. Steinerberger,
		\newblock Poissonian pair correlation in higher dimensions,
		\newblock{\em{J. Number Theory}} 208 (2020), 47--58.
		
		
		
	\end{thebibliography}
\end{document}